\documentclass[smallextended]{llncs}
\usepackage{epsfig}
\usepackage{subfig}
\usepackage{epstopdf}
\usepackage{rotating}
\usepackage[usenames]{color}
\usepackage{multirow}
\usepackage{tikz}
\usetikzlibrary{cd}
\usetikzlibrary{calc}
\usepackage{verbatim}
\usepackage{multicol}
\usepackage{cite}
\usepackage{soul}
\usepackage{thmtools,thm-restate} 
\usepackage{url}
\usepackage[english]{babel}
\usepackage[utf8]{inputenc}
\usepackage{amsmath}
\usepackage[margin=1in]{geometry}
\usepackage[colorinlistoftodos]{todonotes}
\usepackage{dsfont}
\usepackage{xparse}
\usepackage{amsfonts}
\usepackage{psfrag}
\usepackage[all]{xy}
\usepackage{graphicx}
\usepackage{verbatim}
\usepackage{mathtools}
\usepackage{faktor}
\usepackage[normalem]{ulem}
\usepackage{xfrac}
\usepackage{hyperref}

\newcommand{\vertiii}[1]{{\left\vert\kern-0.25ex\left\vert\kern-0.25ex\left\vert #1
\right\vert\kern-0.25ex\right\vert\kern-0.25ex\right\vert}}

\begin{document}

\title{Generalized Permutants and Graph GENEOs}

\titlerunning{Graph GENEOs}

\author{Faraz~Ahmad $^1$,
        Massimo~Ferri $^1$, Patrizio~Frosini $^1$}
        
\authorrunning{F.~Ahmad, M.~Ferri, P.~Frosini}

\institute{$^1$ ARCES and Dept. of Mathematics, Univ. of Bologna, Italy\\
\email{\{faraz.ahmad2, massimo.ferri, patrizio.frosini\}@unibo.it}
}

\maketitle

\begin{abstract}
In this paper we establish a bridge between Topological Data Analysis and Geometric Deep Learning, adapting the topological theory of group equivariant non-expansive operators (GENEOs) to act on the space of all graphs weighted on vertices or edges. This is done by showing how the general concept of GENEO can be used to transform graphs and to give information about their structure. This requires the introduction of the new concepts of generalized permutant and generalized permutant measure and the mathematical proof that these concepts allow us to build GENEOs between graphs. An experimental section concludes the paper, illustrating the possible use of our operators to extract information from graphs. This paper is part of a line of research devoted to developing a compositional and geometric theory of GENEOs for Geometric Deep Learning.
\end{abstract}



\section{Introduction}
\label{intro}

In recent years, the need for an extension of Deep Learning to non-Euclidean domains has led to the development of Geometric Deep Learning (GDL)~\cite{BBLSV17,BrBrCoVe21,bronstein_2022}.
This line of research focuses on applying neural networks on manifolds and graphs, so making available new geometric models for artificial intelligence.
In doing that, GDL uses techniques coming from differential geometry, combinatorics, and algebra. In particular, it largely uses the concepts of group action and
equivariant operator~\cite{AERP19,AnRoPo16,Bengio2013,cohen2016group,Ma12,Ma16,worrall2017harmonic,ZVERP15}, which allow for a strong reduction in the number of parameters involved in machine learning.

Topological Data Analysis (TDA)~\cite{EdLeZo02,CoEdHa07,Car09,EH10,EdMo13} is giving a contribution to the development of GDL, grounding on the use of \emph{non-expansive} equivariant operators~\cite{BFGQ19}. The main idea is to benefit from classical and new results of TDA to study the ``shape'' of the spaces of equivariant operators by employing suitable topologies and metrics. The topological, geometric, and algebraic properties of these spaces have indeed a relevant impact on the identification of the operators that are more efficient for our application purposes.
We stress that the assumption of non-expansivity is fundamental in this framework, since it guarantees that the space of group equivariant non-expansive operators (GENEOs) is compact (and hence finitely approximable), provided that the data space is compact for a suitable topology~\cite{BFGQ19}. We also observe that, from a practical point of view, non-expansivity can be seen as the property of simplifying the metric structure of data. While particular applications may require locally violating this property, we remark that the usual long-term purpose of observers is the one of simplifying the available information by representing it in a much simpler and more meaningful way.

The approach based on TDA has allowed us to start shaping a compositional and topological theory for GENEOs. In particular,
it has been proved that some operations are available to combine GENEOs and obtain other GENEOs, including composition, convex combination, minimization, maximization, and direct product. The compositional theory based on such operations leads us to think of GENEOs as elementary components that could be used to replace neurons in neural networks~\cite{CoFrQu22}.
This new kind of network could be much more transparent in its behavior, because of the intrinsic interpretability of its components.
Modularity is indeed a key tool for interpretability in machine learning, since it can make clear which processes control the behavior of artificial agents. The attention to this property corresponds to the rising interest in the so-called ``explainable deep learning''~\cite{Rudin2019,RoBoDuGa20,longo2020explainable}.

To use GENEOs in applications, we need methods to build such operators for the transformation groups we are interested in.
If we restrict our attention to linear operators, a constructive procedure is available for the case that the functions representing our data have a finite domain $X$. This procedure is based on the concept of ``permutant'', i.e., a set of permutations of $X$ that is invariant under the conjugation action of the equivariance group $G$ we are considering~\cite{CoFrQu22}. While the classical way of building equivariant operators requires integration on the (possibly large) group $G$~\cite{pmlr-v80-kondor18a}, this construction method may be based on a simpler sum computed on a small permutant. We can prove that any linear GENEO can be obtained as a weighted arithmetic mean related to a suitable permutant, provided that the domains of the signals are finite and the equivariance groups transitively act on those domains~\cite{BBBFQ20}. By replacing the weighted arithmetic mean with other normalized symmetric functions, the method based on permutants can be easily extended to the construction of non-linear GENEOs~\cite{CoFrQu22}.

In this paper we explore the possibility of applying the previous ideas to graphs weighted on vertexes or edges, focusing on the important role that these graphs have in GDL~\cite{BBLSV17}. To do that, we illustrate a way of using GENEOs with graphs, defining the concept of \emph{graph GENEO}. Moreover, we extend the definition of permutant to the one of generalized permutant, showing how this new concept allows us to build GENEOs between graphs. Our final purpose is the one of developing a new technique to build operators that can transform graphs according to our needs, and make these operators available as components for application in GDL.

This is the outline of our paper. In Section~\ref{STS} we recall our mathematical setting, based on the concepts of perception pair, GENEO, and permutant. Section~\ref{GPSTS}  introduces and describes the new concepts of generalized permutant and generalized permutant measure, proving that each of them can be used to build a GENEO (Theorems~\ref{GenPermGOTop} and ~\ref{GenPermMeasThm}). In Section~\ref{GraphModel} we introduce the concepts of vertex-weighted/edge-weighted graph GENEO and illustrate our new mathematical model with several examples. A section devoted to experiments (Section~\ref{Experiments}) concludes the paper, showing how graph GENEOs allow us to extract useful information from graphs.

\section{The Set-theoretical Setting}\label{STS}

Let $X$ be a non-empty set, $\Phi$ be a subspace of \[\mathbb{R}_b^X := \{ \varphi : X \to \mathbb{R} \mid \varphi \hspace{0.1cm} \text{is bounded} \}\] endowed with the topology induced by the $L^\infty$ distance 
\[    D_\Phi(\varphi_1, \varphi_2)  := \| \varphi_1 - \varphi_2 \|_\infty = \sup_{x \in X} \lvert\varphi_1(x) - \varphi_2(x)\rvert, \hspace{0.1cm} \varphi_1, \varphi_2 \in \Phi \]
and $G$ be a subgroup of 
\[  \text{Aut}_{\Phi}(X) := \{  g : X \to X \hspace{0.1cm} \mid \text{g is bijective and  }  \varphi \circ g, \varphi \circ g^{-1} \in \Phi, \forall \varphi \in \Phi \} \]

with respect to the composition of functions. 

\begin{definition}\label{DefPercPrTop}
We say that $(\Phi, G)$ is a \textit{perception pair}.
\end{definition}

The elements $\varphi$ of $\Phi$ are often called \textit{measurements}. The fact that $X$ is the common domain of all maps in $\Phi$ will be expressed as dom$(\Phi)=X$.

The space $\Phi$ of measurements endows $X$ and $\text{Aut}_{\Phi}(X)$ (and therefore every subgroup $G$ of $\text{Aut}_{\Phi}(X)$) with topologies induced respectively by the extended pseudometrics \[ D_X(x_1, x_2) := \sup\limits_{\varphi\in\Phi}\lvert\varphi(x_1)-\varphi(x_2)\rvert, x_1, x_2 \in X \] and 
\[D_{\text{Aut}}(f, g) := \sup_{\varphi \in \Phi} D_{\Phi}(\varphi \circ f, \varphi \circ g), \hspace{0.1cm} f, g \in \text{Aut}_{\Phi}(X). \]

It is known that each $g \in \mathrm{Aut}_{\Phi}(X)$ is an isometry of $X$ \cite{BFGQ19}.

\begin{definition}\label{DefGENEOTop}
Let $(\Phi, G)$ and $(\Psi, K)$ be perception pairs with $\text{dom}(\Phi) = X$ and $\text{dom}(\Psi) = Y$, and $T : G \to K$ be a group homomorphism. An operator $F : \Phi \to \Psi$ is said to be a \textit{group equivariant non-expansive operator} (GENEO, for short) \textit{from $(\Phi, G)$ to $(\Psi, K)$ with respect to} $T$ if \[ F(\varphi \circ g) = F(\varphi) \circ T(g), \hspace{0.1cm} \varphi \in \Phi, g \in G, \] and \[ \| F(\varphi_1) - F(\varphi_2) \|_\infty \leq \|\varphi_1 - \varphi_2\|_\infty, \hspace{0.1cm} \varphi_1, \varphi_2 \in \Phi. \]
\end{definition}

For the sake of conciseness, we often write a GENEO as $(F, T) : (\Phi, G) \to (\Psi, K)$.

An operator that satisfies the first condition in this definition is called a \textit{group equivariant operator} (GEO, for short), while one satisfying the second condition is said to be \textit{non-expansive}.

The set $\mathcal{F}_T^{\text{all}}$ of all GENEOs $(F, T) : (\Phi, G) \to (\Psi, K)$, with respect to a fixed homomorphism $T$, is a metric space with the distance function given by 
\[
    D_{\text{GENEO}}(F_1, F_2)  := \sup_{\varphi \in \Phi} D_{\Psi}(F_1(\varphi), F_2(\varphi)),  \hspace{0.1cm} F_1, F_2 \in \mathcal{F}_T^{\text{all}}.
\]

A method to build GENEOs by means of the concept of a permutant is illustrated in \cite{CoFrQu22}. If $G$ is a subgroup of $\mathrm{Aut}_{\Phi}(X)$, then the conjugation map \[ \alpha_g \colon \mathrm{Aut}_{\Phi}(X) \to \mathrm{Aut}_{\Phi}(X), \] given by $f \mapsto g \circ f \circ g^{-1}$, $g \in G$, plays a key role in this technique.

\begin{definition}\label{DefPermTop}
Let $H$ be a finite subset of $\mathrm{Aut}_{\Phi}(X)$. We say that $H$ is a \textit{permutant for $G$} if $H=\emptyset$ or $\alpha_g (H) \subseteq H$ for every $g \in G$; i.e., $\alpha_g (f) = g \circ f \circ g^{-1} \in H$ for every $f \in H$ and $g \in G$.
\end{definition}

\begin{example}\label{ex_uno}
Let $\Phi$ be the set of all functions $\varphi:X=S^1=\{(x,y)\in\mathbb{R}^2 \mid x^2 + y^2 = 1\}\to [0,1]$ that are non-expansive with respect to the Euclidean distances on $S^1$ and $[0,1]$.
Let us consider the group $G$ of all isometries of $\mathbb{R}^2$, restricted to $S^1$. If $h$ is the clockwise rotation of $\ell$ radians for a fixed $\ell\in\mathbb{R}$, then
the set $H=\{h, h^{-1}\}$ is a permutant for $G$.
\end{example}

Other examples of permutants will be given in Example~\ref{PermVM01}, Example~\ref{PermVM02} and Proposition~\ref{SwapGrPermVM}.

We recall the following result.
As usual, in the following we will denote the set of all functions from the set $A$ to the set $B$ by the symbol $B^A$.

\begin{proposition}\label{PermGOTop}
Let $(\Phi, G)$ be a perception pair with $\text{dom}(\Phi) = X$. If $H$ is a nonempty permutant for $G \subseteq \mathrm{Aut}_{\Phi}(X)$, then the restriction to $\Phi$ of the operator $F \colon \mathbb{R}^X \to \mathbb{R}^X$ defined by \[ F(\varphi) := \frac{1}{\lvert H \rvert }\sum_{h\in H} \varphi \circ h \] is a GENEO from $(\Phi, G)$ to $(\Phi, G)$ with respect to $T$, provided that $F(\Phi) \subseteq \Phi$.
\end{proposition}

The reader is referred to \cite{NQTh,BBBFQ20,CoFrQu22} for further details.


\section{Generalized Permutants in the Set-theoretical Setting}\label{GPSTS}

In this paper, we introduce a generalization of the concept of a permutant to the case when we may have distinct perception pairs, and show that the new concept we introduce here too can be used to populate the space of GENEOs.

\begin{definition}\label{DefGenPermTop}
Let $(\Phi, G)$, $\mathrm{dom}(\Phi) = X$ and $(\Psi, K)$, $\mathrm{dom}(\Psi) = Y$ be perception pairs and $T : G \to K$ be a group homomorphism. A finite set $H \subseteq X^{Y}$ of functions $h : Y \to X$ is called a \textit{generalized permutant} for $T$ if $H=\emptyset$ or $g \circ h \circ T(g^{-1})\in H$ for every $h \in H$, and every $g \in G$.
\end{definition}

In this case, we have the following commutative diagram:

\[
\xymatrix{X\ar[r]^{g} & X \\ Y \ar[u]^{h} \ar[r]^{T(g)} & Y \ar[u]_{h'=g \circ h \circ T(g^{-1})}
}
\]

We observe that the map $h\mapsto g \circ h \circ T(g^{-1})$ is a bijection from $H$ to $H$, for any $g\in G$.

Definition~\ref{DefGenPermTop} extends Definition~\ref{DefPermTop} in two different directions. First of all, it does not require that the origin perception pair $(\Phi,G)$ and the target perception pair $(\Psi,K)$ coincide. Secondly, it does not require that the elements of the set $H$ are bijections. In Section~\ref{cycles} we will see how the concept of generalized permutant can be applied.

\begin{example}\label{ex_GP}
Let $X,Y$ be two nonempty finite sets, with $Y\subseteq X$. Let $G$ be the group of all permutations of $X$ that preserve $Y$, and $K$ be the group of all permutations of $Y$. Set $\Phi=\mathbb{R}^X$ and $\Psi=\mathbb{R}^Y$. Assume that $T:G\to K$ takes each permutation of $X$ to its restriction to $Y$. Define $H$ as the set of all functions $h:Y\to X$ such that the cardinality of $\mathrm{Im\ } h$ is smaller than a fixed integer $m$. Then $H$ is a generalized permutant for $T$.
\end{example}

In the following two subsections, we will express two other ways to look at generalized permutants, beyond their definition.
To this end, we will assume that two perception pairs
$(\Phi, G)$, $(\Psi, K)$ and a group homomorphism $T : G \to K$ are given,
with $\mathrm{dom}(\Phi) = X$ and $\mathrm{dom}(\Psi) = Y$.

\subsection{Generalized Permutants as unions of equivalence classes}\label{equiv}

In view of Definition \ref{DefGenPermTop}, we can define an equivalence relation $\sim$ on $X^Y$:

\begin{definition}\label{DefEqRelTop}
Let $h, h' \in X^Y$. We say that $h$ is \textit{equivalent} to $h'$, and write $h \sim h'$, if there is a $g \in G$ such that $h' = g \circ h \circ T(g^{-1})$.
\end{definition}

It is easy to see that $\sim$ is indeed an equivalence relation on $X^Y$.

\begin{proposition}\label{propequivclasses}
A subset $H$ of $X^Y$ is a generalized permutant for $T$ if and only if $H$ is a (possibly empty) union of equivalence classes for $\sim$.
\end{proposition}

\begin{proof}
Assume that $H$ is a generalized permutant for $T$. If $h\in H$ and $h \sim h'\in X^Y$, then the definition of the relation $\sim$ and the definition of generalized permutant imply that $h' \in H$ as well,
and therefore $H$ is a union of equivalence classes for $\sim$. Conversely, if $H$ is a union of equivalence classes for the relation $\sim$, $h\in H$ and $g\in G$,
then $g \circ h \circ T(g^{-1})\in H$, since $g \circ h \circ T(g^{-1})\sim h$. As a consequence, $H$ is a generalized permutant for $T$.
\end{proof}

\subsection{Generalized Permutants as unions of orbits}\label{gpauoo}

The map $\alpha:G\times X^Y\to X^Y$ taking $(g,f)$ to $g \circ f \circ T(g^{-1})$ is a left group action,
since $\alpha(\mathrm{id}_X,f)=\mathrm{id}_X \circ f \circ T(\mathrm{id}_X^{-1})=f$ and
$\alpha(g_2,\alpha(g_1,f))
=\alpha(g_2,g_1 \circ f \circ T(g_1^{-1}))
=g_2\circ (g_1 \circ f \circ T(g_1^{-1}))\circ T(g_2^{-1})
=(g_2\circ g_1) \circ f \circ T((g_2\circ g_1)^{-1})
=\alpha(g_2\circ g_1,f)$. For every $f\in X^Y$, the set $O(f):=\{\alpha(g,f):g\in G\}$ is called the \emph{orbit} of $f$.
By observing that $O(f)$ is the equivalence class of $f$ in $X^Y$ for $\sim$, from Proposition \ref{propequivclasses} the following result immediately follows.

\begin{proposition}\label{proporbits}
A subset $H$ of $X^Y$ is a generalized permutant for $T$ if and only if $H$ is a (possibly empty) union of orbits for the group action $\alpha$.
\end{proposition}

The main use of the concept of generalized permutant is expressed by the following theorem, extending Proposition \ref{PermGOTop}.

\begin{theorem}\label{GenPermGOTop}
Let $(\Phi, G)$, $\mathrm{dom}(\Phi) = X$ and $(\Psi, K)$, $\mathrm{dom}(\Psi) = Y$ be perception pairs, $T : G \to K$ a group homomorphism, and $H$ be a generalized permutant for $T$.
Then the restriction to $\Phi$ of the operator $F : \mathbb{R}^X \to \mathbb{R}^Y$ defined by \[ F(\varphi) := \frac{1}{\lvert H\rvert } \sum_{h \in H} \varphi \circ h \] is a GENEO from
$(\Phi, G)$ to $(\Psi, K)$ with respect to $T$ provided $F(\Phi) \subseteq \Psi$.
\end{theorem}

\begin{proof}
Let $\varphi \in \Phi$ and $g \in G$. Then by the definition of a generalized permutant and the change of variable $h'=g \circ h \circ T(g^{-1})$, we have
\begin{align*}
    F(\varphi \circ g) & := \frac{1}{\lvert H\rvert } \sum_{h \in H} (\varphi \circ g) \circ h \\
    & = \frac{1}{\lvert H\rvert } \sum_{h \in H} \varphi \circ g \circ h \circ T(g^{-1}) \circ T(g) \\
    & = \frac{1}{\lvert H\rvert } \sum_{h' \in H} \varphi \circ h' \circ T(g) \\
    & = F(\varphi) \circ T(g)
\end{align*}
whence $F$ is equivariant.

If $\varphi_1, \varphi_2 \in \Phi$, then
\begin{align*}
    \left\| F(\varphi_1) - F(\varphi_2) \right\|_{\infty} 
    & = \left\| \frac{1}{\lvert H\rvert } \sum_{h \in H} \varphi_1 \circ h - \frac{1}{\lvert H\rvert } \sum_{h \in H} \varphi_2 \circ h \right\|_{\infty} \\
    & = \frac{1}{\lvert H\rvert } \| \sum_{h \in H} (\varphi_1 \circ h - \varphi_2 \circ h) \|_{\infty} \\
    & \leq \frac{1}{\lvert H\rvert } \sum_{h \in H} \| \varphi_1 \circ h - \varphi_2 \circ h \|_{\infty} \\
    & \leq \frac{1}{\lvert H\rvert } \sum_{h \in H} \| \varphi_1 - \varphi_2 \|_{\infty} \\
    & = \frac{1}{\lvert H\rvert } \lvert H\rvert  \| \varphi_1 - \varphi_2 \|_{\infty} \\
    & = \| \varphi_1 - \varphi_2 \|_{\infty}
\end{align*}
whence $F$ is non-expansive, and hence a GENEO.
\end{proof}

\subsection{Generalized permutant measures}\label{GPM}

As shown in \cite{BBBFQ20}, the concept of permutant can be extended to the one of \emph{permutant measure}, provided that the set $X$ is finite.
This is done by the following definition, referring to a subgroup $G$ of the group $\mathrm{Aut}(X)$ of all permutations of the set $X$, and to the perception pair $(\mathbb{R}^X,G)$.

\begin{definition}\cite{BBBFQ20}\label{permutantmeasure}
A finite signed measure ${\mu}$ on $\mathrm{Aut}(X)$ is called a \emph{permutant measure} with respect to $G$ if each subset $H$ of $\mathrm{Aut}(X)$ is measurable and $\mu$ is invariant under the conjugation action of $G$ (i.e., ${\mu}(H)={\mu}(g H g^{-1})$ for every $g\in G$).
\end{definition}

With a slight abuse of notation, we will denote by ${\mu}(h)$ the signed measure of the singleton $\{h\}$ for each $h\in \mathrm{Aut}(X)$.
The next example shows how we can apply Definition \ref{permutantmeasure}.

\begin{example}\label{exS2}
Let us consider the set $X$ of the vertices of a cube in $\mathbb{R}^3$, and the group $G$ of the orientation-preserving isometries of $\mathbb{R}^3$ that take
$X$ to $X$. Set $T=\mathrm{id}_G$. Let $\pi_1,\pi_2,\pi_3$ be the three planes that
contain the center of mass of $X$ and are parallel to a face of the cube.
Let $h_i:X\to X$ be the orthogonal symmetry with respect to $\pi_i$, for
$i\in \{1,2,3\}$. We have that the set $\{h_1,h_2,h_3\}$ is an orbit under
the action expressed by the map $\alpha$ defined in Section \ref{gpauoo}. 
We can now define a permutant measure $\mu$ on $\mathrm{Aut}(X)$ by setting
$\mu(h_1)=\mu(h_2)=\mu(h_3)= c$, where $c$ is a positive real number, and
$\mu(h)=0$ for any $h\in \mathrm{Aut}(X)$ with $h\notin \{h_1,h_2,h_3\}$.
We also observe that while the cardinality of $G$ is $24$, the cardinality
of the support $\mathrm{supp}(\mu):=\{h\in \mathrm{Aut}(X):\mu(h)\neq 0\}$
of the signed measure $\mu$ is $3$.
\end{example}

The concept of permutant measure is important because it makes available the following representation result.

\begin{theorem}\cite{BBBFQ20}\label{mainthm2}
Assume that $G\subseteq \mathrm{Aut}(X)$ transitively acts on the finite set $X$ and $F$ is a map from $\mathbb{R}^X$ to $\mathbb{R}^X$. The map $F$ is a linear group equivariant non-expansive operator 
from $(\mathbb{R}^X,G)$ to $(\mathbb{R}^X,G)$ with respect to the homomorphism $\mathrm{id}_G:G\to G$ if and only if a permutant measure ${\mu}$ exists such that
$F(\varphi)=\sum_{h\in {\mathrm{Aut}(X)}}\varphi \circ h^{-1}\ {\mu}(h)$ for every $\varphi\in {\mathbb{R}^X}$,
and $\sum_{h\in {\mathrm{Aut}(X)}}\lvert \mu(h)\rvert \le 1$.
\end{theorem}

We now state a definition that extends the concept of permutant measure.

\begin{definition}\label{defGPM}
Let $X$, $Y$ be two finite nonempty sets. Let us choose a subgroup $G$ of $\mathrm{Aut}(X)$, a subgroup $K$ of $\mathrm{Aut}(Y)$, and a homomorphism $T:G\to K$. A finite signed measure ${\mu}$ on $X^Y$ is called a \emph{generalized permutant measure} with respect to $T$ if each subset $H$ of $X^Y$ is measurable and ${\mu}\left(g \circ H \circ T(g^{-1})\right)={\mu}\left(H\right)$ for every $g\in G$.
\end{definition}

Definition~\ref{defGPM} extends Definition~\ref{permutantmeasure} in two different directions. First of all, it does not require that the origin perception pair $(\mathbb{R}^X,G)$ and the target perception pair $(\mathbb{R}^Y,K)$ coincide. Secondly, 
the measure $\mu$ is not defined on $\mathrm{ Aut}(X)$ but on the set $X^Y$.

\begin{example}\label{ex_GPM}
Let $X,Y$ be two nonempty finite sets, with $Y\subseteq X$. Let $G$ be the group of all permutations of $X$ that preserve $Y$, and $K$ be the group of all permutations of $Y$. Set $\Phi=\mathbb{R}^X$ and $\Psi=\mathbb{R}^Y$. Assume that $T:G\to K$ takes each permutation of $X$ to its restriction to $Y$. For any positive integer $m$, define $H_m$ as the set of all functions $h:Y\to X$ such that the cardinality of $\mathrm{Im\ } h$ is equal to $m$. For each $h\in H_m$, let us set $\mu(h):=\frac{1}{m \lvert H_m \rvert }$. Then $\mu$ is a generalized permutant measure with respect to $T$.
\end{example}

We can prove the following result, showing that every generalized permutant measure allows us to build a GENEO between perception pairs.

\begin{theorem}\label{GenPermMeasThm}
Let $X$, $Y$ be two finite nonempty sets. Let us choose a subgroup $G$ of $\mathrm{Aut}(X)$, a subgroup $K$ of $\mathrm{Aut}(Y)$, and a homomorphism $T:G\to K$.
If $\mu$ is a generalized permutant measure with respect to $T$, then the map $F_\mu:\mathbb{R}^X\to\mathbb{R}^Y$ defined by setting
$F_\mu(\varphi):=\sum_{f\in X^Y}\varphi \circ f\ {\mu}(f)$ is a linear GEO
from $(\Phi, G)$ to $(\Psi, K)$ with respect to $T$. If $\sum_{f\in X^Y} \lvert \mu(f)\rvert \le 1$, then $F$ is a GENEO.
\end{theorem}

\begin{proof}
It is immediate to check that $F_\mu$ is linear. Moreover, by applying the change of variable $\hat f=g \circ f \circ T(g^{-1})$ and the equality
$\mu\left(g \circ f \circ T(g^{-1})\right)=\mu(f)$, for every $\varphi\in\mathbb{R}^X$ and every $g\in G$ we get
\begin{align*}
F_\mu(\varphi \circ g) 
&=\sum_{f\in {X^Y}}\varphi \circ g \circ f\ {\mu}(f)\\
& =\sum_{f\in {X^Y}}\varphi \circ g \circ f \circ T(g^{-1})\circ T(g)\ {\mu}(g \circ f \circ T(g^{-1}))\nonumber\\
&=\sum_{\hat f\in {X^Y}}\varphi \circ \hat f \circ T(g)\ {\mu}(\hat f)\nonumber\\
&=F_\mu(\varphi) \circ T(g) \nonumber
\end{align*}
since the map $f\mapsto g \circ f \circ T\left(g^{-1}\right)$ is a bijection from $X^Y$ to $X^Y$. This proves that $F_\mu$ is equivariant.

If $\sum_{f\in X^Y}\lvert \mu(f)\rvert \le 1$,
\begin{align*}
\|F_\mu(\varphi)\|_\infty &=\left\|\sum_{f\in {X^Y}}\varphi \circ f\ {\mu}(f)\right\|_\infty\nonumber\\
&\le\sum_{f\in {X^Y}}\left\|\varphi \circ f\right\|_\infty \lvert {\mu}(f)\rvert \nonumber\\
&\le\sum_{f\in {X^Y}}\left\|\varphi \right\|_\infty \lvert {\mu}(f) \rvert \nonumber\\
&=\|\varphi\|_\infty\sum_{f\in {X^Y}} \  \lvert {\mu}(f)\rvert \nonumber\\
&\le\|\varphi\|_\infty.\nonumber
\end{align*}
This implies that the linear map $F_\mu$ is non-expansive, and concludes the proof of our theorem.
\end{proof}

The condition $\lvert \mathrm{supp}(\mu)\rvert \ll \lvert G\rvert $ is not rare in applications (cf., e.g., Example \ref{exS2}) and is the main reason to build GEOs by means of (generalized) permutant measures, instead of using the representation of GEOs as $G$-convolutions and integrating on possibly large groups.

In the following, we will apply the aforementioned concepts to graphs. For the sake of simplicity, we will drop the word ``generalized'' and use the expression ``graph permutant''.


\section{GENEOs on Graphs}\label{GraphModel}

The notions of perception pair, permutant, and GENEO can be easily applied in a graph-theoretic setting. In this section, we develop a model for graphs with weights assigned to their vertices ($\mathrm{vw}$-graphs, for short), and another for graphs with weights assigned to their edges ($\mathrm{ew}$-graphs, for short), often called "weighted graphs" in literature. Our vertex model has implications for the rapidly growing field of graph convolutional neural networks, while the edge model we propose owes its significance to the widely recognized importance of weighted graphs. 

As a graph \cite{bondy1976graph} we shall mean a triple ${\Gamma} =(V_{\Gamma}, E_{\Gamma}, \psi_{\Gamma})$, where $\psi_{\Gamma}$ assigns to each edge of $E_{\Gamma}$ the unordered pair of its end vertices in $V_{\Gamma}$. Since we only consider simple graphs (i.e. with no loops and no multiple edges) we write $e=\{A, B\}$ to mean $\psi_{\Gamma}(e) = \{A, B\}$. Let us recall that an automorphism $g$ of $\Gamma$ is a pair $g = (g_V, g_E)$, where $g_V : V_{\Gamma} \to V_{\Gamma}$ and $g_E : E_{\Gamma} \to E_{\Gamma}$ are bijections respecting the incidence function $\psi_{\Gamma}$. 
The group $\mathrm{Aut}({\Gamma})$ of all automorphisms of $\Gamma$ induces two particular subgroups, here denoted as $\mathrm{Aut}({V_{\Gamma}})$ and $\mathrm{Aut}({E_{\Gamma}})$, of the groups of permutations of $V_{\Gamma}$ and of $E_{\Gamma}$. We represent permutations as cycle products.

For any $k \in \mathbb{N}$, put
\[ \mathbb{N}_k := \{ 1 \le i \le k \mid i \in \mathbb{N} \}. \]
Let a graph $\Gamma = (V_\Gamma, E_\Gamma, \psi_{\Gamma})$ with $n$ vertices and $m$ edges be given. By fixing an indexing of the vertices (resp. edges), we can identify $\mathrm{Aut}({V_{\Gamma}})$ (resp.  $\mathrm{Aut}({E_{\Gamma}})$) with some subgroup of $S_n$ (resp. $S_m$), for the sake of simplicity. Analogously, a real function  defined on $V_{\Gamma}$ (resp. $E_{\Gamma}$) will be represented as an $n$-tuple (resp. $m$-tuple) of real numbers. Anyway, we shall denote vertices (resp. edges) by consecutive capital (resp. lowercase) letters and not by numerical indexes.

In this section, we will consider a space $\Phi_{V_{\Gamma}}$ of real valued functions on $V_\Gamma$, as a subspace of $\mathbb{R}^n$ endowed with the sup-norm $\| \cdot \|_\infty$; i.e., the real valued functions $\varphi \in \Phi_{V_{\Gamma}}$ on the vertex set $V_{\Gamma}$ are given by vectors $\varphi = (\varphi^1, \varphi^2, \cdots, \varphi^n)$ of length $n$. Analogously,
the symbol $\Phi_{E_{\Gamma}}$ will refer to a subspace of $\mathbb{R}^m$ endowed with the sup-norm; i.e., the real valued functions $\varphi \in \Phi_{E_{\Gamma}}$ on the edge set $E_{\Gamma}$ are given by vectors $\varphi = (\varphi^1, \varphi^2, \cdots, \varphi^m)$ of length $m$.

  Let $G$ be a subgroup of the group $\mathrm{Aut}({V_\Gamma})$ (resp. $\mathrm{Aut}({E_\Gamma})$) corresponding to the group of all graph automorphisms of ${\Gamma}$. By the previous convention, the elements of $G$ can be considered to be permutations of the set $\mathbb{N}_n$ (resp. $\mathbb{N}_m$).

\subsection{GENEOs on graphs weighted on vertices}\label{VertexModel}

The concepts of perception pair, GEO/GENEO, and (generalized) permutant can be applied to $\mathrm{vw}$-graphs.

\begin{definition}\label{DefGrPercPrVM}
Let $\Phi_{V_{\Gamma}}$ be a set of functions from $V_\Gamma$ to $\mathbb{R}$ and $G$ be a subgroup of $\mathrm{Aut}({V_{\Gamma}})$. If $(\Phi_{V_{\Gamma}}, G)$ is a perception pair, we will call it a $\mathrm{vw}$-\textit{graph perception pair for} $\Gamma = (V_\Gamma, E_\Gamma, \psi_{\Gamma})$, and will write $\mathrm{dom}(\Phi_{V_{\Gamma}}) = V_\Gamma$.
\end{definition}

\begin{definition}\label{DefGrGOVM}
Let $(\Phi_{V_{\Gamma_1}}, G_1)$ and $(\Phi_{V_{\Gamma_2}}, G_2)$ be two $\mathrm{vw}$-graph perception pairs and $T : G_1 \to G_2$ be a group homomorphism. If $F : \Phi_{V_{\Gamma_1}} \to \Phi_{V_{\Gamma_2}}$ is a {GEO} (resp. {GENEO}) from $(\Phi_{V_{\Gamma_1}}, G_1)$ to $(\Phi_{V_{\Gamma_2}}, G_2)$ with respect to $T$, we will say that $F$ is a $\mathrm{vw}$-\emph{graph GEO} (resp. $\mathrm{vw}$-\emph{graph GENEO}).
\end{definition}

\begin{definition}\label{DefGrPermVM}
Let $(\Phi_{V_{\Gamma_1}}, G_1)$ and $(\Phi_{V_{\Gamma_2}}, G_2)$ be two $\mathrm{vw}$-graph perception pairs and $T : G_1 \to G_2$ be a group homomorphism.
We say that $H\subseteq {V_{\Gamma_1}}^{V_{\Gamma_2}}$ is a $\mathrm{vw}$-\textit{graph permutant for $T$} if $\alpha_g (H) \subseteq H$ for every $g \in G_1$; that is, $\alpha_g (f) = g \circ f \circ T(g^{-1}) \in H$ for every $f \in H$ and $g \in G_1$.
\end{definition}


Let us consider some examples of $\mathrm{vw}$-graph perception pairs, $\mathrm{vw}$-graph GENEOs and $\mathrm{vw}$-graph permutants.

\begin{figure}[ht]
\centering
\includegraphics[width=1.5in]{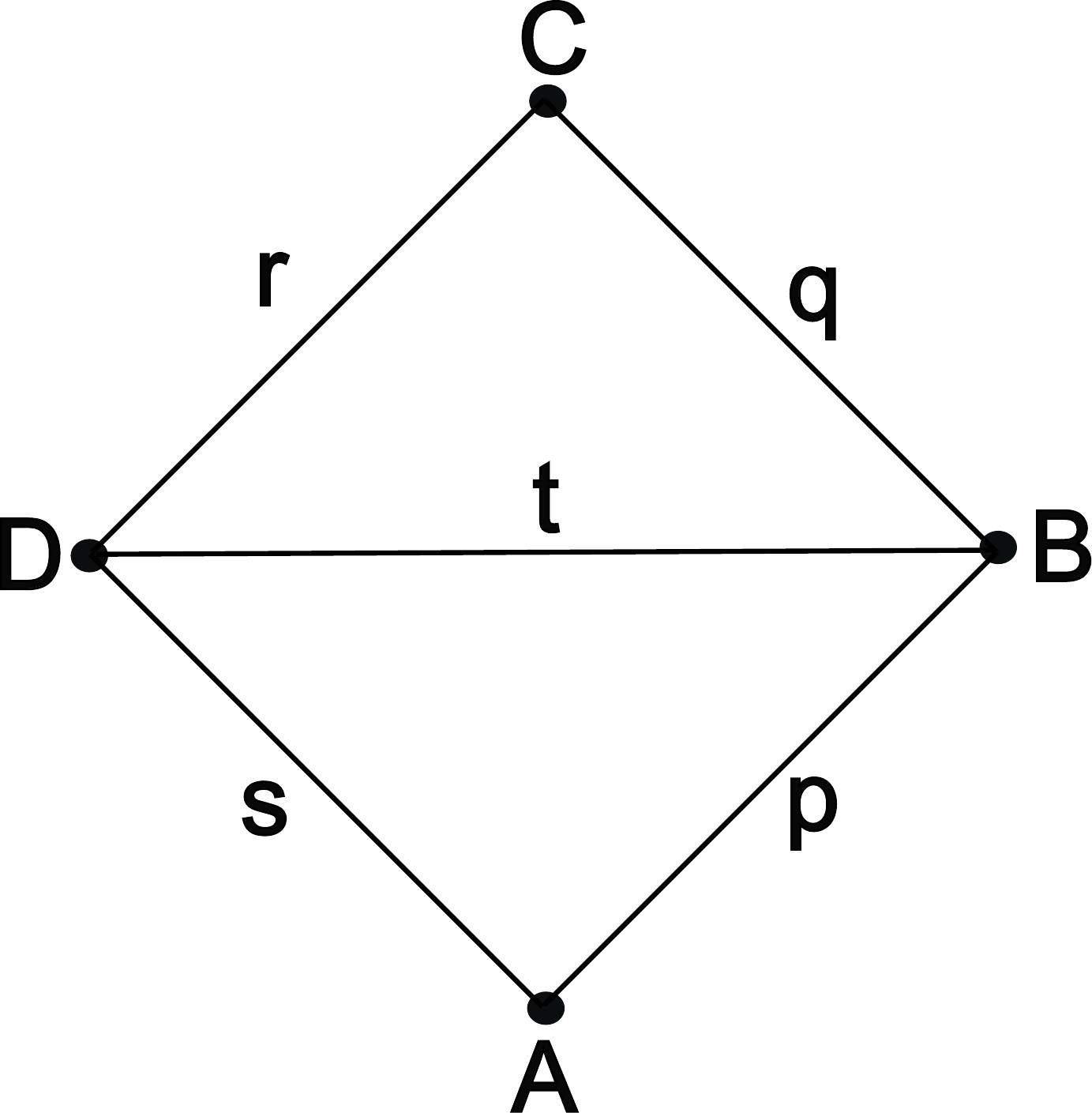}
\caption{The graph of Example~\ref{PnPrVM01}.}
\label{theta}
\end{figure}

\begin{example}\label{PnPrVM01}
Consider the graph $\Gamma = (V_\Gamma, E_\Gamma, \psi_\Gamma)$ with vertex set $V_\Gamma = \{ A, B, C, D \}$ and edge set $E_\Gamma = \big\{ p = \{ A, B\}, q = \{B, C\}, r = \{ C, D\}, s = \{ A, D\}$, $t = \{ B, D\} \big\}$ (see Fig.~\ref{theta}). Its automorphism group $\mathrm{Aut}(V_\Gamma)$ is given by
\[ \mathrm{Aut}(V_\Gamma) = \{\mathrm{id}_{\mathbb{N}_4}, (A, C), (B, D), (A, C)(B, D)\}. \]
Let
\[ G = \{ \mathrm{id}_{\mathbb{N}_4}, \delta =(B, D) \}, \]
and $\Phi_{V_{\Gamma}}$ be the subspace of $\mathbb{R}^4$ given by 
\[ \Phi_{V_{\Gamma}} := \{ \varphi = (\varphi^1, \varphi^2, \varphi^3, \varphi^4) \in \mathbb{R}^4 \mid \varphi^1 + \varphi^3 = 0 \}. \] 
Clearly, $\varphi \circ \delta = (\varphi^1, \varphi^4, \varphi^3, \varphi^2) \in \Phi_{V_{\Gamma}}$ for all $\varphi \in \Phi_{V_{\Gamma}}$; so, $(\Phi_{V_{\Gamma}}, G)$ is a $\mathrm{vw}$-graph perception pair for $\Gamma$.
\end{example}

The next example shows that we can have different perception pairs with the same graph and the same group.

\begin{example}\label{PnPrVM02}

Let $G$ be as in Example \ref{PnPrVM01} and 
\[ \Phi_{V_{\Gamma}} = \{ \varphi = (\varphi^1, \varphi^2, \varphi^3, \varphi^4) \in \mathbb{R}^4 \mid \sum_{i \in \mathbb{N}_4} (\varphi^i)^2 \leq 1 \}. \] Then $(\Phi_{V_{\Gamma}}, G)$ is a $\mathrm{vw}$-graph perception pair.
\end{example}
We can now define a simple class of GENEOs.

\begin{example}\label{GOVM03}
Let $(\Phi_{V_{\Gamma}}, G)$ be as in Example \ref{PnPrVM01} and a map $F$ be defined by 
\[
    F(\varphi) = (\varphi^1 /d_1, \varphi^2 /d_2, \varphi^3 /d_3, \varphi^4 /d_4), \hspace{0.1cm} \varphi \in \Phi_{V_\Gamma}, \hspace{0.1cm} \text{and   } d_1, d_2, d_3, d_4 \in [1, \infty).
\]

If, for all $\varphi = (\varphi^1, \varphi^2, \varphi^3, \varphi^4) \in \Phi_{V_{\Gamma}}$ and $g \in G$, we have $F(\varphi \circ g) = F(\varphi) \circ g$, then
\[\left(\frac{\varphi^1}{d_1}, \frac{\varphi^4}{ d_2}, \frac{\varphi^3}{d_3}, \frac{\varphi^2}{d_4}\right) = \left(\frac{\varphi^1}{d_1}, \frac{\varphi^4}{d_4}, \frac{\varphi^3}{d_3}, \frac{\varphi^2}{d_2}\right)\]
whence $d_2 = d_4$; and the requirement that $F(\varphi) \in \Phi_{V_{\Gamma}}$ entails $d_1 = d_3$. 

Moreover, 
\[
    \| F(\varphi_1) - F(\varphi_2) \|_\infty 
     \leq \frac{1}{\min\{d_1, d_2\}} \| \varphi_1 - \varphi_2 \|_\infty \le \| \varphi_1 - \varphi_2 \|_\infty
\]
for all $\varphi_1 = (\varphi^1_1, \varphi^2_1, \varphi^3_1, \varphi^4_1), \ \varphi_2 = (\varphi^1_2, \varphi^2_2, \varphi^3_2, \varphi^4_2) \in \Phi_{V_{\Gamma}}$, whence $F$ is non-expansive.

Therefore, the map $F$ defined above is a $\mathrm{vw}$-graph GENEO if and only if $d_1 = d_3$ and $d_2 = d_4$.
\end{example}

We now prepare for the first instances of graph permutants in Examples~\ref{PermVM01} and~\ref{PermVM02}.

\begin{example}\label{PnPrVM04}
Let $\Gamma = (V_\Gamma, E_\Gamma, \psi_\Gamma)$ be the cycle graph $C_4$ with $V_\Gamma = \{ A, B, C, D \}$. Its automorphism group is given by 
\[
    \mathrm{Aut}(V_\Gamma) = \{  \mathrm{id}_{\mathbb{N}_4}, \alpha = (A, B, C, D), \alpha^2, \alpha^3, (A, C),  (B, D), (A, B)(C, D), (A, D)(B, C) \}
\]
and \[ G = \{ \mathrm{id}_{\mathbb{N}_4}, \alpha, \alpha^2, \alpha^3 \} \] is a subgroup of $\mathrm{Aut}(V_\Gamma)$.

If $\Phi_{V_{\Gamma}}$ is the same as in Example \ref{PnPrVM01}, then $(\Phi_{V_{\Gamma}}, G)$ is not a $\mathrm{vw}$-graph perception pair. However, if we define 
\[
    \Phi_{V_{\Gamma}} = \{ \varphi = (\varphi^1, \varphi^2, \varphi^3, \varphi^4) \in \mathbb{R}^4 \mid \varphi^1 + \varphi^3 = 0 = \varphi^2 + \varphi^4 \}
\]
then $(\Phi_{V_{\Gamma}}, \mathrm{Aut}(V_\Gamma))$, and therefore $(\Phi_{V_{\Gamma}}, G)$, are $\mathrm{vw}$-graph perception pairs.
\end{example}

\begin{example}\label{PnPrVM06}
Let $G$ be as in Example \ref{PnPrVM04} and \[ \Phi_{V_{\Gamma}} = \{ \varphi \in \mathbb{R}^4 \mid \|\varphi\|_{\infty} \leq 1 \}. \] Then $(\Phi_{V_{\Gamma}}, G)$ is a $\mathrm{vw}$-graph perception pair.
\end{example}

\begin{example}\label{PermVM01}
Let $G$ be as in Example \ref{PnPrVM04} and 
\[
H = \{ h_1 = (A, B)(C, D), h_2 = (A, D)(B, C) \} \subseteq \mathrm{Aut}({V_{\Gamma}}).
\]
Then $H$ is a $\mathrm{vw}$-graph permutant for $T = \mathrm{id}_G$.
\end{example}

\begin{example}\label{PermVM02}
Let $\Gamma$ be as in Example \ref{PnPrVM04} and \[ G = \{ \mathrm{id}_{\mathbb{N}_4}, \alpha^2, (A, B)(C, D), (A, D)(B, C) \} \] be the Klein 4-group contained in $\mathrm{Aut}({V_{\Gamma}})$. If \[ H = \{ (A, C), (B, D) \} \subseteq \mathrm{Aut}({V_{\Gamma}}) \] then $H$ is a $\mathrm{vw}$-graph permutant for $T = \mathrm{id}_G$.
\end{example}

As usual, in the following we will denote by $K_n$ the complete graph on $n$ vertices.

\begin{proposition}\label{SwapGrPermVM}
Let $\Gamma := K_n$ and $H \subseteq G = \mathrm{Aut}({V_{\Gamma}}) \cong S_n$ be the set of all transpositions of $V_\Gamma 
$. Then $H$ is a $\mathrm{vw}$-graph permutant for $T = \mathrm{id}_G$.
\end{proposition}

\begin{proof}
Let $f \in H$ and $g \in G$; we show that $g \circ f \circ g^{-1} \in H$. Let us put $f := (A, B)$ for some $A, B \in V_\Gamma$, $C := g(A)$ and $D := g(B)$. Then \[ g \circ f \circ g^{-1}(C) = g \circ f (A) = g(B) = D \] \[ g \circ f \circ g^{-1}(D) = g \circ f (B) = g(A) = C. \]

While if $L \in V_\Gamma$ is different from both $C$ and $D$, then as $g$ is bijective, $g^{-1}(L) \neq g^{-1}(C) = A$ and $g^{-1}(L) \neq g^{-1}(D) = B$. We thus have \[ g \circ f \circ g^{-1}(L) = g \circ g^{-1}(L) = L \] whence $g \circ f \circ g^{-1} = (C, D) \in H$, as required.
\end{proof}

As stated in Theorem \ref{GenPermGOTop}, the concept of a $\mathrm{vw}$-graph permutant can be used to define $\mathrm{vw}$-graph GENEOs.

\begin{example}
Let $(\Phi_{V_{\Gamma}}, G)$ be the same as in Example \ref{PnPrVM06} and $H$ be the same as in Example \ref{PermVM01}. Set $F(\varphi) = \frac{1}{2} ( \varphi \circ h_1 + \varphi \circ h_2 )$. Then $F(\Phi_{V_{\Gamma}}) \subseteq \Phi_{V_{\Gamma}}$; therefore by Theorem \ref{GenPermGOTop}, $F$ is a $\mathrm{vw}$-graph GENEO.
\end{example}

\subsection{GENEOs on graphs weighted on edges}\label{EdgeModel}

The concepts of perception pair, GEO/GENEO, and (generalized) permutant can be applied to $\mathrm{ew}$-graphs as well.

\begin{definition}\label{DefGrPercPrEM}
Let $\Phi_{E_{\Gamma}}$ be a set of functions from $E_\Gamma$ to $\mathbb{R}$ and $G$ be a subgroup of $\mathrm{Aut}({E_{\Gamma}})$.
If $(\Phi_{E_{\Gamma}}, G)$ is a perception pair, we will call it an $\mathrm{ew}$-\textit{graph perception pair for} $\Gamma = (V_\Gamma, E_\Gamma, \psi_{\Gamma})$, and will write $\mathrm{dom}(\Phi_{E_{\Gamma}}) = E_\Gamma$.
\end{definition}

\begin{definition}\label{DefGrGOEM}
Let $(\Phi_{E_{\Gamma_1}}, G_1)$ and $(\Phi_{E_{\Gamma_2}}, G_2)$ be two $\mathrm{ew}$-graph perception pairs and $T : G_1 \to G_2$ be a group homomorphism. If $F : \Phi_{E_{\Gamma_1}} \to \Phi_{E_{\Gamma_2}}$ is a {GEO} (resp. {GENEO}) from $(\Phi_{E_{\Gamma_1}}, G_1)$ to $(\Phi_{E_{\Gamma_2}}, G_2)$ with respect to $T$, we will say that $F$ is an $\mathrm{ew}$-\emph{graph GEO} (resp. $\mathrm{ew}$-\emph{graph GENEO}).
\end{definition}

\begin{definition}\label{DefGrPermEM}
Let $(\Phi_{E_{\Gamma_1}}, G_1)$ and $(\Phi_{E_{\Gamma_2}}, G_2)$ be two $\mathrm{ew}$-graph perception pairs and $T : G_1 \to G_2$ be a group homomorphism.
We say that $H\subseteq {E_{\Gamma_1}}^{E_{\Gamma_2}}$ is an $\mathrm{ew}$-\textit{graph permutant for $T$} if $\alpha_g (H) \subseteq H$ for every $g \in G_1$; that is, $\alpha_g (f) = g \circ f \circ T(g^{-1}) \in H$ for every $f \in H$ and $g \in G_1$.
\end{definition}

The group $\mathrm{Aut}(\Gamma)$ of all graph automorphisms of a graph $\Gamma$ induces a particular subgroup $\mathrm{Aut}({E_\Gamma})$ of the group $S_{m}$ of all permutations of $E_\Gamma$. The elements of $\mathrm{Aut}({E_\Gamma})$ can be considered to be those permutations of $E_\Gamma$ that directly correspond to the permutations of $V_\Gamma$ defining all graph automorphisms of $\Gamma$.

If $\Gamma = K_n$, the group $\mathrm{Aut}(\Gamma)$ is isomorphic to $S_n$, and we have \[  S_n \cong \mathrm{Aut}({V_\Gamma}) \cong \mathrm{Aut}({E_\Gamma}) \subseteq S_m. \] Therefore, we will consider $\mathrm{Aut}(\Gamma)$ and $\mathrm{Aut}({E_\Gamma})$ to be the same in this case.

Let us consider some examples of perception pairs and GENEOs in the context of $\mathrm{ew}$-graphs.

\begin{example}\label{PnPrEM01}
Let $\Gamma = K_4 = (V_\Gamma, E_\Gamma, \psi_\Gamma)$ with 
\[
  V_{K_4} = \{A, B, C, D \} \]
\[  E_{K_4} = \big\{  p = \{A, B\}, q = \{B, C\}, r = \{A, C\},  s = \{A, D\}, t = \{B, D\}, u = \{C, D\} \big\}
\]
(see Fig.~\ref{lettereK4}), and consider the group $G = \{ \mathrm{id}_{E_\Gamma}, \hspace{0.1cm} \delta = (r \hspace{0.1cm} s)(q  \hspace{0.1cm} t) \} \subseteq \mathrm{Aut}({E_\Gamma})$ together with the space $\Phi_{E_\Gamma} = \{ \varphi = (\varphi^1, \varphi^2, \varphi^3, \varphi^4, \varphi^5, \varphi^6) \mid \varphi^1 + \varphi^6 = 0 \} \subseteq \mathbb{R}^m$ of the functions with opposite values on the two edges fixed by the elements of $G$. Clearly, $\varphi \circ \delta \in \Phi_{E_\Gamma}$, and $(\Phi_{E_\Gamma}, G)$ is an $\mathrm{ew}$-graph perception pair.
\end{example}

\begin{figure}[ht]
\centering
\includegraphics[width=1.3in]{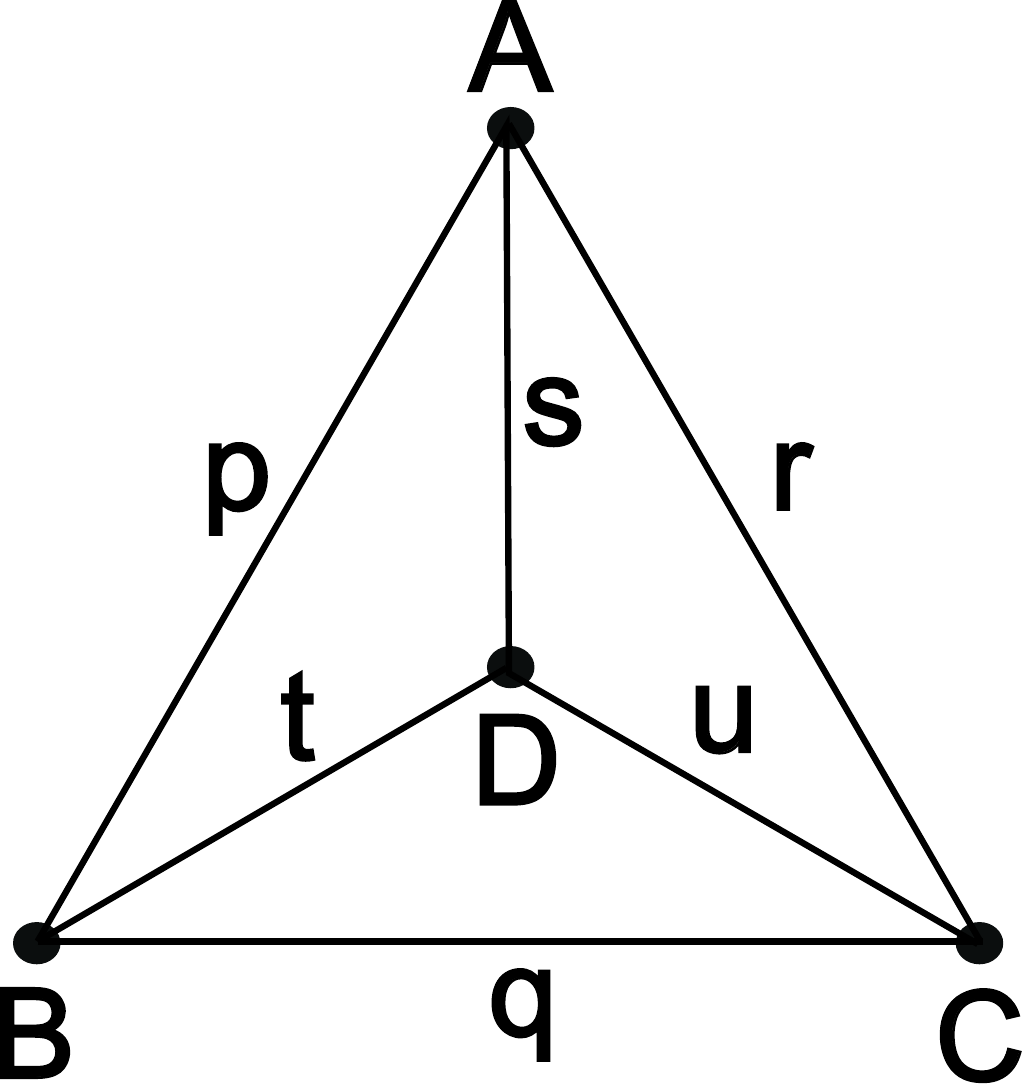}
\caption{The complete graph $K_4$.}
\label{lettereK4}
\end{figure}

\begin{example}\label{GOEM01}
Let $(\Phi_{E_\Gamma}, G)$ be as in Example \ref{PnPrEM01} and consider the map $F$ defined by 
\begin{align*}
   F(\varphi) := (& \varphi^1/d_1, \varphi^2/d_2, \varphi^3/d_3, \varphi^4/d_4, \varphi^5/d_5, \varphi^6/d_6), \\
   & \varphi \in \Phi_{E_\Gamma}, \hspace{0.1cm} \text{and} \hspace{0.1cm} d_i \geq 1, \hspace{0.1cm} \forall i \in \mathbb{N}_6.
\end{align*}

In order that $F(\varphi) \in \Phi_{E_\Gamma}$ we should have $d_1 = d_6$, and the requirement that $F$ be equivariant with respect to $G$ entails $d_3 = d_4$ and $d_2 = d_5$.

Moreover, a simple computation shows that 
\begin{align*}
   \| F(\varphi_1) - F(\varphi_2) \|_\infty 
   & \leq \frac{1}{\min\{d_1, d_2, d_3\}} \| \varphi_1 - \varphi_2 \|_\infty \\  & \le \| \varphi_1 - \varphi_2 \|_\infty
\end{align*}
for all $\varphi_1 = (\varphi^i_1 \hspace{0.1cm} / \hspace{0.1cm} i \in \mathbb{N}_6), \varphi_2 = (\varphi^i_2 \hspace{0.1cm} / \hspace{0.1cm} i \in \mathbb{N}_6) \in \Phi_{E_\Gamma}$, whence $F$ is non-expansive.

Therefore, the map $F$ defined above is an $\mathrm{ew}$-graph GENEO if and only if $d_1 = d_6, d_2 = d_5$, and $d_3 = d_4$.
\end{example}

\section{Experiments}\label{Experiments}

We illustrate the model of Section \ref{EdgeModel} and show how graph GENEOs allow us to extract useful information from graphs. This can be done by ``smart forgetting'' of differences: by some sort of average, but keeping the same dimension of the space of functions (as in Sect.~\ref{subK4}) or by dimension reduction (as in Sect.~\ref{cycles}).

\subsection{Subgraphs of $K_4$}\label{subK4}

The choice of a permutant determines how different functions are mapped to the same ``signature'' by the corresponding GENEO. In this subsection, we consider functions on the edge set of a complete graph $K_n$, taking values that are either 0 or 1; this means that each such a function identifies a subgraph of $K_n$. A GENEO will, in general, produce functions that can have any real value, so not representing subgraphs anymore. With the aim of getting equal results for ``similar'' subgraphs, we have chosen as a permutant the set of edge permutations produced by swapping two vertices in any possible way.

Let $\Gamma$ be the complete graph $K_4$ (Fig.~\ref{lettereK4}) with $\Phi_{E_\Gamma} := \mathbb{R}^6$. We have \[ S_4 \cong \mathrm{Aut}(K_4) \cong \mathrm{Aut}(E_{K_4}) \subseteq S_6. \] 
The subset 
\[
  H := \{(q,r)(s,t), (p,q)(s,u), (p,t)(r,u), (p,r)(t,u), (p,s)(q,u), (q,t)(r,s)\}
\]
of $G = \mathrm{Aut}(E_{K_4})$ consisting of permutations of $E_{K_4}$ induced by all transpositions of $V_{K_4}$ is an $\mathrm{ew}$-graph permutant for $T = \mathrm{id}_G$. Therefore, the operator $F \colon \mathbb{R}^6 \to \mathbb{R}^6$ defined by
\[ F(\varphi) := \frac{1}{6} \sum_{h \in H} \varphi \circ h \]
is an $\mathrm{ew}$-graph GENEO.

\begin{figure}[ht]
\centering
\includegraphics[width=6in]{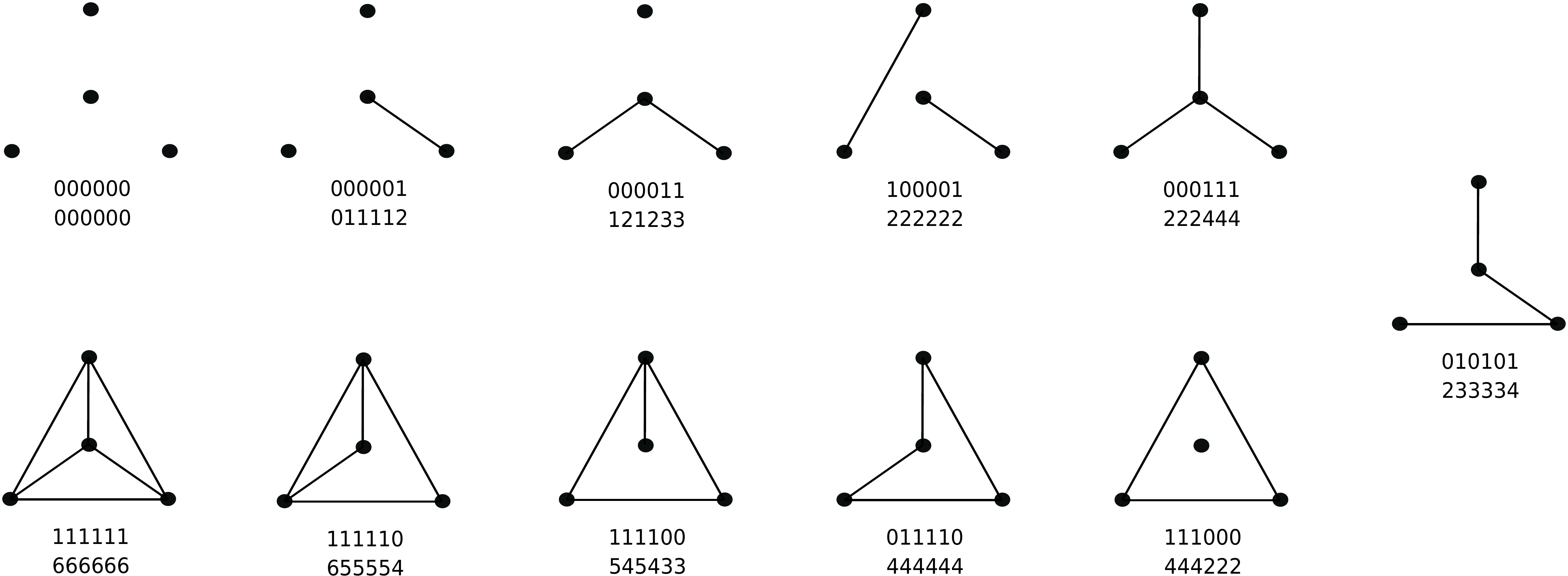}
\caption{The subgraphs of $K_4$ up to isomorphisms, the 6-tuples representing each of them (above) and their $F_4$-codes (multiplied by 6, below).}
\label{K4}
\end{figure}

Subgraphs of $K_4$ can be represented by elements of
\[\Phi_4 := \big\{ \varphi = (\varphi^1, \cdots, \varphi^6) \in \Phi_{E_{K_4}} \mid \varphi^r \in \{0, 1\}, r \in \mathbb{N}_6 \big\} \]
and the restriction $F_4$ of $F$ to $\Phi_4 \subseteq \Phi_{E_{K_4}}$ can be used to draw meaningful comparisons between them (see Fig.~\ref{K4}).

\begin{definition}
We say that the image $F_4(\varphi)$ is an $F_4-$\textit{code} for the subgraph $\varphi \in \Phi_4$ of $K_4$.
\end{definition}

\begin{definition}
We say that an $F_4-$\textit{code} $c_1$ is $F_4-$\textit{equivalent} to an $F_4-$\textit{code} $c_2$, and write $c_1 \sim_4 c_2$, if $c_2$ is the result of a permutation of $c_1$.
\end{definition}

Clearly, $\sim_4$ is an equivalence relation.

\begin{definition}
We say that $\varphi^{'} := (\varphi^6, \cdots, \varphi^1)$ is the \textit{reversal} of $\varphi := (\varphi^1, \cdots, \varphi^6) \in \Phi_{E_{K_4}}$.
\end{definition}

\begin{definition}
Let $\varphi_1, \varphi_2 \in \Phi_{E_{K_4}}$. We say that $\varphi_1$ and $\varphi_2$ are complementary if $\varphi_1 + \varphi_2 = (1, \cdots, 1)$
\end{definition}

We wrote a simple program to compute all $F_4$-codes and found that
\begin{enumerate}
    \item Naturally enough, isomorphic subgraphs have $F_4$-equivalent codes. Therefore, in some cases, it suffices to consider only the 11 non-isomorphic subgraphs of $K_4$.
    \item Complementary subgraphs have complementary codes.
    \item There is only one case, upto graph isomorphisms, in which non-isomorphic subgraphs of $K_4$ have $F_4$-equivalent codes: $\varphi_1 := (1,1,1,0,0,0)$ and $\varphi_2 := (0,0,0,1,1,1)$ with $F_4(\varphi_1) := (4,4,4,2,2,2)/6$ and $F_4(\varphi_2) := (2,2,2,4,4,4)/6$. In this case, the graphs are complementary as well, which explains why we have equivalent codes despite the graphs being non-isomorphic. Moreover, $\varphi_1$ and $\varphi_2$ are reversals of each other, and so are the corresponding codes.
    \item If $\varphi_1 \in \Phi_4$ is a reversal of $\varphi_2 \in \Phi_4$, then $F_4(\varphi_1)$ is a reversal of $F_4(\varphi_2)$.
\end{enumerate}

We wrote a program to compute $F_5-$codes for the 34 non-isomorphic subgraphs of $K_5$ as well and found that they were never $F_5-$equivalent. A similar statement holds for the complete graph $K_3$.

\subsection{Graph GENEOs for $C_6$ and $C_3$}\label{cycles}

A more drastic way of quotienting differences away is dimension reduction of the space of functions. In this subsection, we use the generalized notion of permutant (Sect.~\ref{GPSTS}), by mapping the edges of a small, auxiliary graph to the edges of the graph of interest. Note that we have great freedom, in that we are not bound to stick to graph homomorphisms.

Let $X := (V_X, E_X, \psi_X)$ be the cycle graph $C_6$ (see Fig.~\ref{C6C3}) with

\[ V_X := \{ A, B, C, D, E, F \}, \] 
\[
   E_X := \{ a = \{A,B\}, b = \{B,C\}, c = \{C,D\},  d = \{D,E\}, e = \{E,F\}, f = \{A,F\} \}
\]
and $Y := (V_Y, E_Y, \psi_Y)$ be the cycle graph $C_3$ with \[ V_Y := \{ G, H, I \}, \] \[ E_Y := \{ g = \{G,H\}, h = \{H,I\}, i = \{G,I\} \}. \]

Their automorphisms groups respectively are the dihedral groups \[ D_6 := \{ \alpha, \beta \mid \alpha^6 = \beta^2 = (\beta \alpha)^2 = 1 \}, \] \[ D_3 := \{ \gamma, \delta \mid \gamma^3 = \delta^2 = (\delta \gamma)^2 = 1 \}, \] where $\alpha := (a,b,c,d,e,f), \beta := (a,f)(b,e)(c,d), \gamma := (g,h,i)$, and $\delta := (g,i)$.

Let us put $\Phi_{E_X} := \mathbb{R}^6, \ \Phi_{E_Y} := \mathbb{R}^3$; also put $G := \mathrm{Aut}(E_X) = D_6$ and $K := \mathrm{Aut}(E_Y) = D_3$ and consider the group homomorphism $T : G \to K$ given by $T(\alpha) := \gamma$ and $T(\beta) := \delta$.

There are 216 functions $p : E_Y \to E_X$ and the equivalence class of each is an $\mathrm{ew}$-graph permutant $H_p$ (Sect.~\ref{equiv}). For the sake of conciseness, we will denote the function $p := \{ (g\mapsto e_1), (h\mapsto e_2), (i\mapsto e_3) \}$ simply by $p := \mathrm{e_1 e_2 e_3}$. For example, $p := \{ (g\mapsto c), (h\mapsto a), (i\mapsto f) \}$ will be written as $p := \mathrm{caf}$.

The $\mathrm{ew}$-graph permutants $H_p, p \in E_X^{E_Y}$ are of four possible sizes:

\begin{enumerate}
    \item There is only 1  $\mathrm{ew}$-graph permutant with 2 elements. It corresponds to the function $\mathrm{aec}$.
    \item There is only 1  $\mathrm{ew}$-graph permutant with 4 elements. It is induced by $\mathrm{bfd}$.
    \item There are 5 $\mathrm{ew}$-graph permutants with 6 elements each that correspond to the functions $\mathrm{aaa}, \mathrm{abc}, \mathrm{ace}, \mathrm{add}$, and $\mathrm{afb}$.
    \item There are 15 $\mathrm{ew}$-graph permutants with 12 elements each that correspond to the functions $\mathrm{aab}$, $\mathrm{aac}$, $\mathrm{aad}$, $\mathrm{aae}$, $\mathrm{aaf}$, $\mathrm{abd}$, $\mathrm{acb}$, $\mathrm{acd}$, $\mathrm{adb}$, $\mathrm{adc}$, $\mathrm{baa}$, $\mathrm{bad}$, $\mathrm{bca}$, $\mathrm{bce}$, and $\mathrm{bdb}$. 
\end{enumerate}

\begin{figure}[ht]
\centering
\includegraphics[width=2.4in]{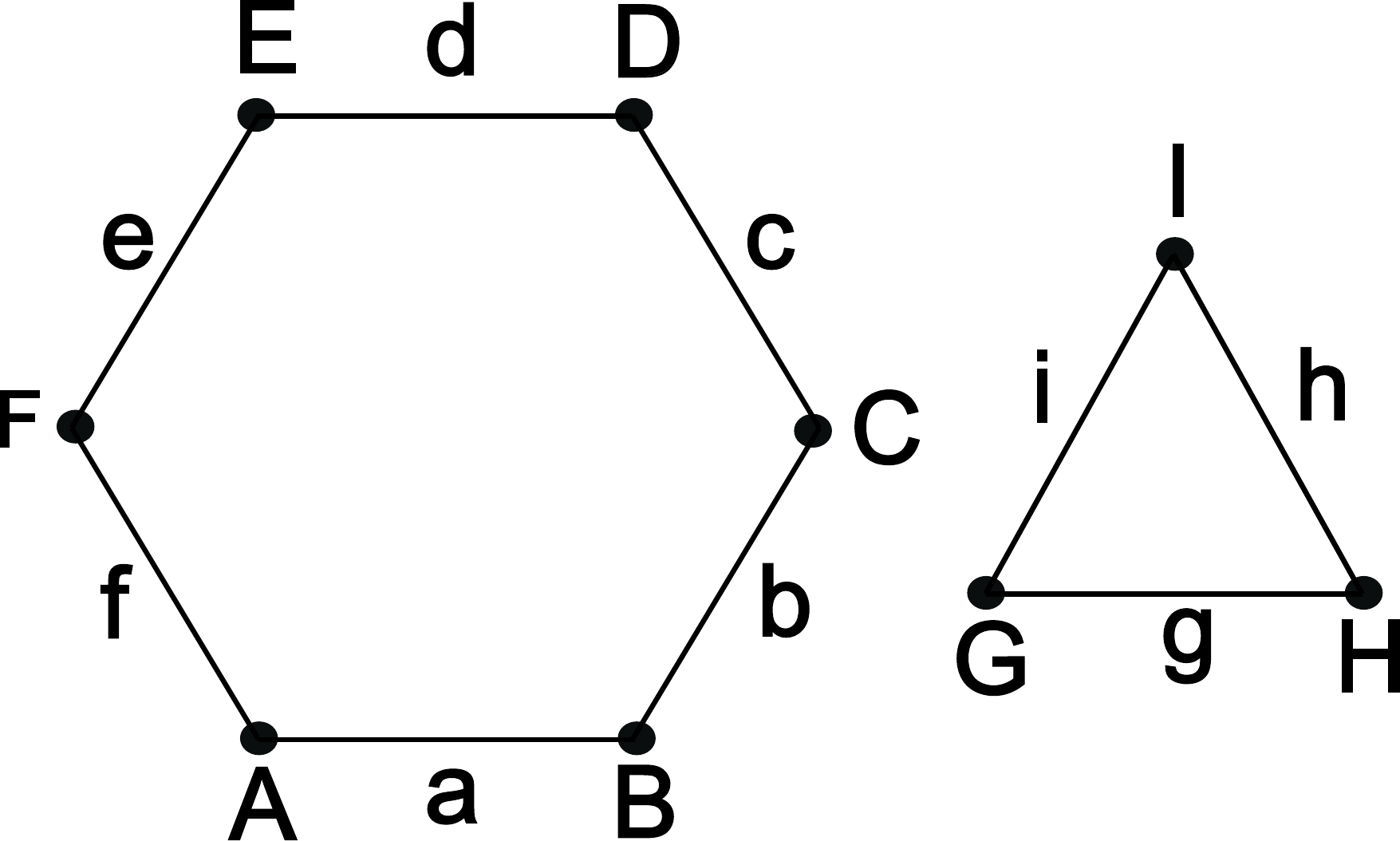}
\caption{The cycle graphs $C_6$ and $C_3$.}
\label{C6C3}
\end{figure}

Considering only the weights in $\{0, 1\}$, we wrote programs for computing the $\mathrm{ew}$-graph GENEOs corresponding to the functions $\mathrm{aec}$ and $\mathrm{bfd}$. Similar computations can be made for the rest of the functions listed above. This detailed analysis on particular functions raised a number of questions and conjectures that we plan to study in the near future.

\section{Discussion}
\label{Dis}
GENEOs represent the possibility to embed observers into Machine Learning processes, one step towards Explainable AI~\cite{Rudin2019,RoBoDuGa20,longo2020explainable}. On the other hand, they are precisely defined mathematical tools; as such, they were first conceived in a topological setting for TDA. In the same environment, permutants turn out to be effective gears for the production of GENEOs. The ever-growing use of graphs in data representation and Geometrical Deep Learning \cite{hu2020open,latouche2015graphs,makarov2021survey,li2022multiphysical,BBLSV17} motivated the present extension to the graph-theoretical domain, mirroring an analogous generalization already carried out in Topological Persistence \cite{bergomi2021beyond,bergomi2022steady}.

The extension to graphs was performed here on two different lines: the first sees data as functions defined on the vertices of a graph; this is perhaps the most common use of graphs in ML. The second development line deals with functions defined on the graph edges; this is best suited for processing filtered graphs and for comparing different graphs with the same vertex set. In both lines, we applied a generalized definition of permutant, introduced in this same paper.

The examples studied here are just meant to show some realizations of the introduced abstract concepts: our goal was mainly to establish a solid mathematical background. Extensions to digraphs and possibly hypergraphs will follow. We are confident that experimenters, interested in concrete applied problems, will make good use of the flexibility and modularity of the theory, for translating their viewpoints and biases into GENEOs and permutants also in the graph-theoretical setting.

The C++ programs used here are available at the repository
\href{https://gitlab.com/patrizio.frosini/graph-geneos}{https://gitlab.com/patrizio.frosini/graph-geneos}.

\section*{Acknowledgement}
This research has been partially supported by INdAM-GNSAGA.  




\bibliographystyle{spmpsci}
\bibliography{Gong}

\end{document}